\documentclass[11pt]{amsart}
\usepackage[usenames,dvipsnames]{pstricks}
\usepackage[mathscr]{eucal}
\usepackage{amsfonts,amsmath,amssymb,amsthm,amscd
	,amsxtra}
\usepackage{enumerate,verbatim,textcomp}
\usepackage[all,2cell,ps]{xy}
\usepackage[notcite, notref]{}
\usepackage[pagebackref]{hyperref}
\usepackage{calc,color}
\theoremstyle{plain} 

\newtheorem{thm}{Theorem}[section]
\newtheorem{cor}[thm]{Corollary}
\newtheorem{lem}[thm]{Lemma}
\newtheorem{nota}[thm]{Notation}
\newtheorem{prop}[thm]{Proposition}
\newtheorem{defn}[thm]{Definition}

\newtheorem{exam}[thm]{Example}
\newtheorem{rem}[thm]{Remark}

\newtheorem{res}[thm]{Result}

\newcommand{\Ann}{\mbox{Ann}\,}

\newcommand{\Hom}{\mbox{Hom}\,}
\newcommand{\Ext}{\mbox{Ext}\,}
\newcommand{\e}{\mbox{e}\,}

\newcommand{\Spec}{\mbox{Spec}\,}
\newcommand{\Max}{\mbox{Max}\,}

\newcommand{\Ass}{\mbox{Ass}\,}

\newcommand{\tr}{\mbox{tr}\,}

\newcommand{\depth}{\mbox{depth}\,}
\newcommand{\vdim}{\mbox{vdim}\,}

\newcommand{\q}{\mbox{q}\,}

\newcommand{\Soc}{\mbox{Soc}\,}

\newcommand{\GGL}{\mbox{GGL}\,}

\newcommand{\AGL}{\mbox{AGL}\,}

\newcommand{\h}{\mbox{ht}\,}

\newcommand{\E}{\mbox{E}}

\renewcommand{\H}{\mbox{H}}

\newcommand{\redu}{\mbox{r}\,}

\newcommand{\re}{\mbox{red}\,}

\newcommand{\lo}{\longrightarrow}
\newcommand{\su}{\subseteq}

\newcommand{\fa}{\mathfrak{a}}

\newcommand{\fm}{\mathfrak{m}}
\newcommand{\fp}{\mathfrak{p}}
\newcommand{\fy}{\mathfrak{y}}

\newcommand{\fq}{\mathfrak{q}}
\newcommand{\fn}{\mathfrak{n}}
\newcommand{\fx}{\mathfrak{x}}

\begin{document}
	\bibliographystyle{amsplain}
		\title[rings with canonical reductions]
		{rings with canonical reductions}
		\author[M. Rahimi]{Mehran Rahimi$^1$}
	\address{$^{1}$ Faculty of Mathematical Sciences and Computer,
		Kharazmi University, Tehran, Iran.}
	\email{mehran\_rahimy@yahoo.com}

\begin{abstract}
We study the class of one dimensional
	 Cohen-Macaulay local rings with canonical reductions, i.e. admit  canonical ideals which are reductions of the maximal ideals, show that it contains the class of almost Gorenstein rings and study characterizations for rings obtained by idealizations or by numerical semigroup rings to have canonical reductions.
\end{abstract} 
\maketitle
\section{Introduction}

 In \emph{\cite{BF}}, Barucci and Fr\"{o}berg, (1997), introduced the class of almost Gorenstein rings which is a subclass of one-dimensional analytically unramified rings. This concept has been generalized to Cohen-Macaulay local rings, studied in details by Goto et.al \cite{GTP}, \cite{GTT} and \cite{GTTU}.  The theory is still developing (see \cite{CDKM} and \cite{GGL})  introducing new classes of rings such as 2-almost Gorenstein local (2-$\AGL$ for short) and generalized Gorenstein local ($\GGL$ for short) rings. 
 
 The aim of this paper is to study Cohen-Macaulay local rings possessing canonical ideals.  
 Over a Cohen-Macaulay  local ring $(R, \fm, k)$ of dimension $d$, a maximal Cohen-Macaulay module $\omega$  is called a {\it canonical module} of $R$ if $\dim_k\Ext_R^i(k , \omega)=\delta_{id}$. An ideal of $R$, isomorphic to a canonical module, is called a {\it canonical ideal} of $R$. A local ring admits a canonical module if and only if it is a homomorphic image of a Gorenstein ring. Such a ring $R$ admits a canonical ideal if and only if $R_\fp$ is Gorenstein for all associated prime ideals $\fp$ of $R$. (For more informations on canonical modules, see \cite{HK} and \cite[Section 3.3]{BH}).
  
 In Section 2,  for a fixed $R$, we denote by $\mathfrak{C}_R$, the set of all canonical ideals of $R$ and investigate it in  details. Here is a result of this kind.
 
 \begin{res}\em{(Theorem \ref{MaxCanIdeal})} Let $I$ and $J$ belong to $\mathfrak{C}_R$. \begin{itemize}
 		\item [(a)] If $I \subset J$, then $I = xJ$ for some $x \in \fm$ where $xR = (I:J)$.
 		\item [(b)] If $x \in I$ is a regular element, then $xJ = yI$ for some $y \in J$.
 		\item [(c)] Let $\dim R = 1$. If $\re(\fm) = t$ then $\fm^{t+1}$ does not contain any element of $\emph\Max({\mathfrak{C}_R})$, where $\re(\fm)$ denotes the reduction number of $\fm$.
 	\end{itemize}
 \end{res}
As (a) shows, the {\it maximal canonical ideals} of $R$, $\Max(\mathfrak{C}_R)$, determines all other canonical ideals. So, non-maximal canonical ideals  are contained in  $\fm^2$. If $R$  have minimal multiplicity, then maximal canonical ideals does not belong to $\fm^2$ which is a partial converse of (c).
 
 In Section 3, we concentrate on one-dimensional Cohen-Macaulay local rings admitting a canonical ideal. When this is the case, canonical ideals are $\fm$-primary and so, it is natural to ask when a canonical ideal is a reduction for $\fm$. It will be shown first that a canonical ideal is a reduction for $\fm$ when $R$ is an almost Gorenstein ring. By \cite{GTP}, a one-dimensional Cohen-Macaulay local ring $(R,\fm)$ is called {\it almost Gorenstein} if $\e_{\fm}^0(R) \leq \redu(R)$. Proposition \ref{SFE} guarantees that almost Gorenstein rings have a canonical reduction. This motives us to prepare next definition.
 \begin{defn}\em{(Definition \ref{DCR1}) An ideal $I$ is called a {\it canonical reduction} of $R$ if $I$ is a canonical ideal of $R$ and is a reduction of $\fm$.}
 \end{defn}
 Its clear that all Gorenstein local rings, with infinite residue fields, admit  canonical reductions. Let $R$ be a non-Gorenstein ring. We set $S_{\mathfrak{C}_R} = \{ \ell(R/K) \ | \ K \in \mathfrak{C}_R \ \}$. Proposition \ref{CanRedMaxIdeal}  shows that a canonical reduction $J$, if exists, is a  maximal canonical ideal which is not contained in $\fm^2$ and  $\ell(R/J)\in\min(S_{\mathfrak{C}_R})$.

 Our next result emphasizes on  singular one-dimensional Cohen-Macaulay local rings. 
 \begin{res} \emph{(Proposition \ref{Min(S_c)=2})}
 	Let $(R,\fm)$ be a singular one-dimensional Cohen-Macaulay local ring. Then
 	\begin{itemize}
 		\item [(a)] If $K$ is a canonical ideal of $R$, then $\fm K = \fm (K : \fm)$.
 		\item [(b)] If $\min(S_{\mathfrak{C}_R}) = 2$ then $R$ has a canonical reduction.
 	\end{itemize}
 \end{res}
 Part (a) is proved before for some special canonical ideals (see\cite[Lemma 3.6]{CHV}). In (b) we deal with the minimum value of $S_{\mathfrak{C}_R}$. Note that the maximal ideals can not be canonical ideals when $R$ is not regular. 
 
  It is also proved in (a) that canonical ideals can not be integrally closed in one-dimensional singular local rings. A well known fact that canonical ideals of a normal local ring are integrally closed, motives us to investigate the theory for rings with dimensions greater than $1$.
 
 \begin{res}\emph{(Lemma \ref{LemInt} and Corollary \ref{CorInt})}
 	Let $\dim R > 1$ and $K \in \mathfrak{C}_R$. Then
 	\begin{itemize}
 		\item $K$ is integrally closed if and only if $K$ is a radical ideal.
 		\item If all canonical ideals of $R$ are integrally closed, then $R$ is an integrally closed ring.
 	\end{itemize} 
 \end{res}

Next we try to find rings that does not have a canonical reduction. Proposition \ref{R/K=2MinMult} show that non-almost Gorenstein rings with minimal multiplicity does not have a  canonical reduction.

The concepts of generalized Gorenstein and nearly Gorenstein rings are defined in \cite{GGL} and \cite{HTS}, respectively. Here we prove that nearly Gorenstein rings have canonical reductions. Also we provide a criterion for a generalized Gorenstein ring to admit a canonical reduction. These results will be generalized to rings with arbitrary positive dimensions in section 5. 

We close section 3, by providing a characterization for a numerical semigroup ring to admit a canonical reduction (Theorem \ref{NCR}). Finally, we study the  behavior of certain Ulrich ideals in a ring with canonical reduction.
\begin{res}\emph{(Proposition \ref{URed})}
	Let $(R, \fm)$ be a local ring with $\dim R = 1$, $R/\fm$ is infinite, and $R$ admits an Ulrich ideal $I$ with $\mu(I) > 2$. Then the following statements are equivalent.
	\begin{itemize}
		\item[(i)] $R$ has a canonical reduction.
		\item[(ii)] $R$ is generalized Gorenstein with respect to $I$, and $I$ is a  reduction of $\fm$.
	\end{itemize} 
\end{res}  

Section 4 is motivated by \cite[Section 6]{GTP}. The main purpose is to find some characterizations for rings which admit canonical reductions and provide some examples by means of idealizations. In this section the ring $R$ is a one-dimensional Cohen-Macaulay local ring with canonical module $K_R$. It is well-known that if $M$ is maximal Cohen-Macaulay $R$-module, then $R\ltimes M$ is a Cohen-Macaulay local ring with the same residue field of $R$. 
\begin{res}\emph{(Proposition \ref{IProp1})}\label{r} Let $R$ be a one-dimensional Cohen-Macaulay local ring with canonical module $K_R$.
	Then the following statements are equivalent.
		\begin{enumerate}[\em(i)]
			\item $R\ltimes M$ admits a canonical reduction $I\ltimes L$, for some submodule $L$ of $M$.
			\item $M \cong \emph\Hom_R(I , K_R)$ and $I$ is a reduction of $\fm$.
		\end{enumerate}
\end{res}
As an immediate corollary, $R\ltimes \Hom_R(\fm , K_R)$ admits a canonical reduction, which means every Cohen-Macaulay local ring $R$, of dimension one, is  homomorphic image of a ring $S$ such that $S$ admits a canonical reduction and $\min(S_{\mathfrak{C}_S}) = 2$.

Next result shows nice behavior of rings with canonical reductions via idealization.
\begin{res}\emph{(Theorem \ref{IdealRed})}
	Under the conditions in Result \ref{r}, the followings statements are equivalent.
	\begin{enumerate}[\em(i)]
		\item $R$ has a canonical reduction.
		\item $R\ltimes R$ has a canonical reduction.
		\item $R \ltimes \fm$ has a canonical reduction.
	\end{enumerate}
\end{res}
We end section 4 by constructing some examples of rings with canonical reductions of arbitrary large canonical index in  Remark \ref{Irem}.

Section 5 deals with rings with canonical reductions whose dimensions are $>1$. First comes the definition.
\begin{defn}
	\em{Assume that $(R,\fm)$ is a $d$-dimensional Cohen-Macaulay local ring admitting a canonical ideal. The ring $R$ is said to  have a  \emph{canonical reduction} $K$ if $K$ is a canonical ideal of $R$ and  there exists an equimultiple  ideal $I$ such that $\h(I) = d-1$ and $K + I$ is a reduction of $\fm$. }
\end{defn}
It is proved that every nearly Gorenstein ring admits a canonical reduction (Proposition \ref{CanRedNG}). Also Proposition \ref{AGCR} gives a criterion for a generalized Gorenstein ring to have a canonical reduction. In particular, it shows that every $d$-dimensional almost Gorenstein ring admits a canonical reduction.
After a discussion about equimultiple ideals, we study canonical reduction via linkage. Here is one of the results.
\begin{res}\emph{(Proposition \ref{LinkCanRed})}
	Let $R$ admits a canonical reduction and $I$ and $J$ be zero-hight ideals of $R$ such $I$ is linked to $J$. If $R/I$ is Cohen-Macaulay and $\mu(J) = 1$  then $R/I$ has a canonical reduction. 
\end{res} 
 Note that if $R$ is Gorenstein and $I$ and $J$ be as above then it is known that $R/I$ is Gorenstein.

Finally, we end the paper by investigating  canonical reductions via flat local homomorphisms.
\begin{res}\emph{(Theorem \ref{canredflat})} Assume that $\varphi : (R,\fm) \lo (S, \fn)$ is a flat local ring homomorphism of generically Gorenstein Cohen-Macaulay local rings of positive dimensions $d$ and $n$, respectively. Assume further that $S/\fm S$ is Gorenstein. Consider the following two conditions.
	\begin{itemize}
		\item [(a)] $R$ has a canonical reduction and $\e_{\fn}^0(S) = \e_{\fn}^0(S/\fm S)\e_{\fm}^0(R)$.
		\item [(b)] $S$ has  a canonical reduction.
	\end{itemize}
	Then the following statements hold true.
	\begin{itemize}
		\item [(i)] If $n = d$, then (a)$\Rightarrow$(b). The converse holds if $R/\fm$ is infinite.
		\item [(ii)] If $n > d$, then  (a)$\Leftrightarrow$(b)  if $R/\fm$ is infinite and $S/\fm S$ is a regular ring.
	\end{itemize}	
\end{res}

Throughout $(R,\fm, k)$ is a Noetherian local  ring  with maximal ideal $\fm$ and all $R$-modules are finitely generated. Set $\e_{\fm}^0(R)$ as the multiplicity  of $R$ with respect to $\fm$. For an $R$-module $M$, we denote $\ell(M)$ for the length of $M$ as $R$-module, and set the Cohen-Macaulay type of $M$ as $\redu(M): = \ell(\Ext_R^{\dim(M)}(R/\fm , M))$.

		 



	\section{The set of canonical ideals of a ring}

 In this chapter,  we show that all canonical ideals can be expressed by the set of maximal canonical ideals. In fact, we will prove that each canonical ideal of $R$ is equal to $xI$ for some maximal canonical ideal $I$ and an $R$-regular element $x$. 
		
	Throughout this section, $(R, \fm)$ is a  Cohen-Macaulay local ring of positive dimension admitting a canonical ideal. Let $\mathfrak{C}_R$ be the set of all canonical ideals of $R$ and set $ \Max(\mathfrak{C}_R)$ as the set of all maximal elements of $\mathfrak{C}_R$ with respect to inclusion.

\begin{lem}\label{mimu}
	Assume that  $(R, \fm)$ is a one-dimensional Cohen-Macaulay local ring. Let $I \subsetneq J$ be  ideals of $R$ such that $I$ is a canonical ideal of $R$. Then $\mu(I : J) = \emph\redu(J)$, where $\emph\redu(J)=\ell(\emph\Ext^1_R(R/\fm, J))$ is the type of $J$ as an $R$-module.
\end{lem}

\begin{proof}
	Set $\E = \E_{R}(R/\fm)$ the injective hull of $R/\fm$. Applying $\Hom_R(- , \E)$ on the exact sequence $0 \lo J/I \lo R/I \lo R/J \lo 0$ gives the exact sequence $$0 \lo \Hom_R(R/J , \E) \lo \Hom_R(R/I , \E) \lo \Hom_R(J/I , \E) \lo 0.$$
	As $R/I$ is a Gorenstein ring of dimension $0$, $\Hom_R(R/J , \E) \cong T/I$ for some proper ideal $T$ of $R$ and $\Hom_R(J/I , \E) \cong R/T$. By Gorensteinness of $I$, $T = I : (I : T) = (I : J)$ and 	so $\Hom_R(J/I , \E) \cong R/(I : J)$. On the other hand, the sequence $0 \lo I \lo J \lo J/I \lo 0$ implies the exact sequence $$0 \lo  \Hom_R(R/\fm , J/I) \lo \Ext_R^1(R/\fm , I) \lo \Ext_R^1(R/\fm , J) \lo \Ext_R^1(R/\fm , J/I) \lo 0.$$
	As $\redu(I) = 1$, we get the Bass numbers $\mu^0(J/I) = 1$ and so $\mu^1(J/I) = \redu(J)$. Applying $\Hom_R( - , \E)$ on the  minimal injective resolution  
	$0 \lo J/I \lo \E \lo \E^{\redu(J)} \lo \cdots $ of $J/I$, implies the exact sequence  $ \cdots \lo \widehat{R}^{\redu(J)} \lo \widehat{R} \lo R/(I : J) \lo 0$, which is the minimal free resolution of $\Hom_R(J/I , \E) \cong R/(I : J)$ over $\widehat{R}$. Therefore, we have $\mu(I : J) = \mu_{\widehat{R}}((I : J)\widehat{R}) = \redu(J)$.
\end{proof}

\begin{rem}\label{CNM}\emph{
		Let $R$ be a one-dimensional Cohen-Macaulay local ring which is not regular. In \cite[Corollary 2.4]{GHI}, it is shown that the maximal ideal of $R$ is not a canonical ideal of $R$. As a consequence, Lemma \ref{mimu} shows that $\fm^i$ is not canonical ideal of $R$ for each $i>0$. }
\end{rem} 
In the following, it is shown that every canonical ideal of $R$ which is not contained in $\fm^2$ is a maximal canonical ideal of $R$. First we state a lemma. 

\begin{lem}\label{suploc}
	Assume that $(R,\fm)$ is a $d$-dimensional Cohen-Macaulay local ring and that $I$ is a height one ideal of $R$. Assume that $a \in I$ is a superficial element which is $R$-regular. If $P \in \emph\Ass(R/I)$ such that $I_P$ is cyclic then $I_P = aR_P$.  
\end{lem}
\begin{proof}
	Let  $I_P = xR_P$. By \cite[Lemma 8.5.3]{SH}, $I^{n+1} : a = I^n$ for $n \gg 0$. Localizing at $P$, we get $x^{n+1}R_P : aR_P = x^nR_P$ which means $aR_P$ is a superficial element of $xR_P$. Therefore $aR_P = xR_P = I_P$.
\end{proof} 

\begin{defn}\label{equimultiple}\em{	Let $R$ be a local ring and $I$ be an ideal of $R$. Then $I$ is called {\it equimultiple} if $I$ has a reduction $J$ such that $\mu(J) = \h(I)$. In this case $J$ is a minimal reduction of $R$.}
		\end{defn}
		 Note that $\fm$-primary ideals are equimultiple ideals.
	
\begin{thm}\label{MaxCanIdeal} 
	Assume that $(R, \fm)$ is  a non-Gorenstein Cohen-Macaulay local ring of dimension $d>1$ with infinite residue field. Let $I$ be a canonical ideal of $R$.
	\begin{itemize}
		\item [(a)] If  $I\not\in\emph\Max({\mathfrak{C}_R})$ then  for  any $T\in{\mathfrak{C}_R}$ containing $I$ properly, $I = xT$ for some regular element $x\in\fm$. In particular, $I \su \fm^2$.
		\item [(b)] (See \cite[Satz 2.8]{HK}.) If $K \in \mathfrak{C}_R$ and $x \in K$ is an $R$-regular element, then there exists $y \in I$ such that $yK = xI$. In this situation, $x$ is a reduction of $K$ if and only if $y$ is a reduction of $J$. 
		
		\item [(c)]  Let $\dim R = 1$. If $\emph\re(\fm) = t$ then $\fm^{t+1}$ does not contain any element of $\emph\Max({\mathfrak{C}_R})$. In particular, if $R$ has minimal multiplicity, then non-maximal canonical ideals of $R$ are exactly the canonical ideals which contained in $\fm^2$. 
	\end{itemize} 
\end{thm}  
\begin{proof}
	(a). Assume that $T$ is a  canonical ideal of $R$ properly containing $I$. First let $d=1$. As $\redu(T) = 1$, by Lemma \ref{mimu}, 
	$\mu_R(I : T)  = 1$. One may write $(I : T) = aR$ for some regular element $a \in \fm$. As  $I \subseteq (I : T) = aR$, there is an ideal $J$ such that $I = aJ$ and consequently $J$ is a canonical ideal of $R$.  As $R$ is non-Gorenstein, $J$ is a proper ideal and so $I\subseteq \fm^2$. As $I$ is Gorenstein ideal of $R$, $T = (I : a) = (aJ : a) = J$. Note that in this case, we do not need the assumption that $R/\fm$ is infinite. 
	
	Now let $d>1$. Choose $a \in (I:T)$ to be a superficial element which is $R$-regular. As $R/I$ is Gorenstein, $T = I : (I : T) \su (I : a)$. We claim that $T = (I : a)$. To proceed the proof, we will show that $(I : a)R_P \su T_P$ for all $P \in \Ass(R/T)$. By Lemma \ref{suploc}, if $P \in \Ass(R/(I:T))( = \Ass(T/I))$ then $IR_P = xTR_P$ for some $x \in PR_P$  which depends  on $P$. By proof of one-dimensional case and Lemma \ref{suploc}, $(I : T)R_P = aR_P$ and $I_P = aTR_P$ for all $P \in \Ass(R/(I:T))$. Now let $P \in \Ass(R/T)$. There is two possible cases. 
		
		(i) $P \in \Ass(R/(I:T))$. In this case $(I : a)R_P = (aT_P : aR_P) = T_P \su T_P$.
		
		(ii) $P \notin \Ass(R/(I:T))$. Then $T_P = I_P$ and $(I : a)R_P = (T_P : aR_P)$. If $a \notin P$ then we are done. Assume contrary that $a \in P$. Then $P \in \Ass(R/aR) \su \Ass(R/(I:T))$ which is a contradiction by assumption.  
		
		Hence $T= (I : a)$ and $(I : T) = aR + I$. If $(I : T)$ is not equimultiple then there exists a superficial element $a \in I$ for $(I : T)$. Hence, by Lemma \ref{suploc}, we must have $(I : T) = aR + I = I$ which means $T = R$ and $R$ is Gorenstein, a contradiction. Therefore $(I : T)$ is equimultiple, and so $\Ass(R/aR) = \Ass(R/(I:T))$. As $(I:T)_P = aR_P$ for all $P \in \Ass(R/(I:T))$, this means $(I:T) = aR$ and $I \su aR$. Set $I = aJ$. Then $T = (I : a) = (aJ : a) = J$ and the proof is complete.

	(b). As $x \in K$, $xI \su K$ and so by (a) there exists $y \in \fm$ such that $xI = yK$. 
	
	 (c). Let $ T \su \fm^{t+1}$ be a maximal element of ${\mathfrak{C}_R}$. As $R/\fm$ is infinite, there exists $x \in \fm$ such that $\fm^{t+1} = x\fm^t$. Therefore $T = xJ$ for some canonical ideal $J$ of $R$, which contradicts with $T \in \Max({\mathfrak{C}_R})$.
\end{proof}


The following example shows that residue rings of  maximal canonical ideals may not have the equal lengths.

\begin{exam}\emph{(See \cite[Example 3.13]{GTP}). 
		Let $R = k[[t^3, t^4, t^5]] \su k[[t]]$. Then $\e_{\fm}^0(R) = 3$,  $R$ has minimal multiplicity with  $\fm^2 = t^3\fm$, $R$ is almost Gorenstein, and $K = (t^3, t^4)$ is a canonical ideal of $R$ with $\ell(R/K) = 2$. Hence $K$ is a maximal canonical ideal of $R$. Set $I = (t^4, t^5)$, then $t^6K = t^6(t^3, t^4) = t^5(t^4, t^5) = t^5I$ which means $I$ is a canonical ideal of $R$. As $t^5 \in I$, one has $I \nsubseteq \fm^2$ which means $I$ is a maximal canonical ideal by  Proposition \ref{MaxCanIdeal}, while  the exact sequence  $$0 \lo R/(t^3, t^5) \overset{\times t^3}{\lo} R/I \lo R/(t^3, t^5) \lo 0$$ implies that $\ell(R/I) = 4$.}
\end{exam}


	 \section{Rings with canonical reductions}
	 
	  In this section, after presenting the definition of canonical reductions for one-dimensional Cohen-Macaulay local rings, we explore the rings which admit canonical reductions. The first result concerns almost Gorenstein rings.

	  \begin{prop}\label{SFE}
	  	Assume that $(R,\fm)$ is a one dimensional almost Gorenstein local ring. Then,  for each $R$-regular element $x$, there exists a canonical ideal $J$ for $R$, depending on $x$, such that $x \in J$ and $J\fm\subseteq xR$. In particular, if $R$ has infinite residue field then $R$ admits a canonical ideal which is a reduction of $\fm$.
	  \end{prop}
	  \begin{proof}  In case  $R$ is Gorenstein, the ideal  $J := xR$ satisfies the required property for any regular element $x$. So we assume that $R$ is an  almost Gorenstein ring which is not Gorenstein. Assume that $K$ is a canonical ideal of $R$ and that $\omega\in K$ is a reduction of $K$. By \cite[Theorem 3.11]{GTP}, we have $\fm K = \omega\fm$.  Assume that $x$ is a regular element of $R$. As $xK \su \omega R$ we have $xK = \omega J$ for some ideal $J$ of $R$. Hence we have $\omega J\fm = xK\fm = x\omega\fm $ which means $J\fm = x\fm $ and $J/xR$ is a vector space over $R/\fm$. Note that $J \cong K$, therefore $J$ is a canonical ideal of $R$ contains $x$.
	  	
	  	In particular, if $R/\fm$ is infinite, there is a superficial element $x \in \fm$ which is a reduction of $\fm$. Now the result follows by above discussion.
	  \end{proof}

Proposition \ref{SFE}  motivates us to investigate the class of  one-dimensional Cohen-Macaulay local rings admitting a canonical ideal which is a reduction of the maximal ideal. This class contains one-dimensional Gorenstein local rings with infinite residue fields.
	   	Now we are ready for the main definition of this section.
\begin{defn}\label{DCR1}\emph{
	  		Assume that $(R, \fm)$ is a one-dimensional Cohen-Macaulay local ring that admits a canonical ideal. A {\it canonical reduction} of $R$ is a canonical ideal $I$ of $R$ such that $I$ is a reduction of $\fm$. }
 \end{defn}
	
 Let $R$ be a one-dimensional Gorenstein ring with infinite residue field. Then $R$ admits a reduction element $x$ which is also  $R$-regular. As $xR$ is a canonical ideal of $R$, every one-dimensional Gorenstein local ring with infinite residue field admits a canonical reduction.  Also Proposition \ref{SFE} shows that every one-dimensional almost Gorenstein local ring with infinite residue field admits a canonical reduction. 
 
   There is no general information in case the residue field $R/\fm$ is finite. Note that \cite[Remark 2.10]{GTP} indicates that the ring
 	$R = k[[X, Y, Z]]/(X, Y)\cap(X, Z)\cap(Y, Z)$, where $k[[X, Y, Z]]$ is the power series ring over field $k$, 
 	 is almost Gorenstein in the sense of  \cite[Definition 3.1]{GTP}.  The ideal $I = (x+y, y+z)$ is a canonical ideal of $R$ and $\fm^2 = (x^2, y^2, z^2) = I\fm$. This means that $I$ is a canonical reduction of $R$, while $I$ has no principle reduction when $k = \mathbb{Z}/2\mathbb{Z}$.

Let $R$ be a Cohen-Macaulay local ring, admitting a canonical ideal. Remind that $\mathfrak{C}_R$ is denoted for the set of all canonical ideals of $R$ and that  $S_{\mathfrak{C}_R}: = \{ \ell(R/K) \ | \ K \in \mathfrak{C}_R \ \}$. Note that $1 \notin S_{\mathfrak{C}_R}$ by Remark \ref{CNM}.

  \begin{prop}\label{CanRedMaxIdeal}
  	Assume that $(R,\fm)$ is a one-dimensional Cohen-Macaulay local ring with infinite residue field and that $R$ admits a canonical reduction $K$. Then
  	\begin{itemize}
  		\item [(a)]  for each canonical ideal $J$ of $R$, $\ell(R/J) \geq \ell(R/K)$ and equality holds if and only if $J$ is a reduction of $\fm$. In  particular $\min(S_{\mathfrak{C}_R}) = \ell(R/K)$.
  		\item [(b)] $\ell(R/K) \leq \emph\e_{\fm}^0(R) - \emph\redu(R) + 1$. Equality holds if and only if $R$ is almost Gorenstein.
  	\end{itemize}   	
  \end{prop}
  \begin{proof}
  	(a). By the uniqueness of canonical modules there exists an $R$-isomorphism  $f : K \lo J$. Assume that $x \in K$ is a reduction of $\fm$ and set $y = f(x)$. Then $K/xR \cong J/yR$ and we have
  	\begin{equation}\label{e1}
  		\begin{array}{lllll}
  		\ell(R/K) &= \ell(R/xR)-\ell(K/xR)\\ &=\e_{\fm}^0(R)-\ell(K/xR)\\ &\leq \ell(R/yR) - \ell(J/yR)\\ &= \ell(R/J).
  		\end{array}
  		\end{equation} 
  	Therefore $\ell(R/K) \leq \ell(R/J)$ and equality holds   if and only if $\e_{\fm}^0(R) = \ell(R/yR)$ which, by \cite[Theorem 11.3.1]{SH}, $y \in J$ and $x$ is a reduction element for $\fm$.
  	
  	(b). Using the second equality of (\ref{e1}) and the fact that $\ell(K/xR) \geq \mu(K/xR) = \redu(R) - 1$ imply the claimed inequality.	Equality holds if and only if $ \ell(K/xR) = \redu(R) - 1$, i.e. $R$ is almost Gorenstein.
  \end{proof}


By \cite[Proposition 4.1]{GGHV}, if $R$ is a $d$-dimensional generically Gorenstein Cohen-Macaulay local ring of dimension $d\geq 1$, all canonical ideals of $R$ have the same reduction number, i.e. the reduction number of a canonical ideal is an invariant denoted by $\rho(R)$ and is called the {\it canonical index} of $R$ (see \cite[Definition 4.2]{GGHV}).  It is proved, in \cite[Theorem 3.7]{GTP}, that $R$ is Gorenstein if and only if $\rho(R) = 1$.   The nearest case to Gorenstein rings by means of $\rho(R)$ is the case when $\rho(R) = 2$ which is the object of the following remark. Note that if $R$ is a one-dimensional non-Gorenstein almost Gorenstein local ring then $\rho(R) = 2$ by \cite[Theorem 3.16]{GTP}.

 \begin{rem}\label{remrho=2}
 	\emph{Assume that $(R, \fm)$ is  a one-dimensional Cohen-Macaulay local ring with infinite residue field admitting a canonical ideal $K$. Let $\rho(R) = 2$ and set $x\in K$ as a reduction of $K$. Then the ideal $\fa = (x : K)$ is independent from a particular canonical ideal $K$ and a reduction element $x$ by \cite[Theorem 2.5]{CDKM}. Moreover, by \cite[Proposition 2.3]{CDKM}, $K^2 \su xR$ which means $K^2 = xJ$ for some  proper ideal $J$ of $R$. We list some statements with sketches of their proofs. 
 		\begin{itemize}
 			\item [(i)] \emph{$K : (x : K) = J$}. 
 				By \cite[Proposition 2.3]{CDKM}, when $\rho(R) = 2$ we have $\fa = (xK : K^2) = (xK : xJ) = (K : J)$, which means $K : (x : K) = J$.
 			\item [(ii)] \emph{$J^2 = K^2 = xJ = KJ$}.  As  $K^2 = xJ$, so $x^2J^2 = K^4 = xK^3 = x^2K^2$ which means $J^2 = K^2 = xJ$.				
 			\item [(iii)] \emph{ $(K:\fm^i) \su J$ if and only if $\fa \su \fm^i + K$. In particular $(K : \fm) \su J \subset \fm$}. If $J = K$, then $R$ is Gorenstein, a contradiction. Therefore, As $(K : J) \su \fm$,  we get $(K:\fm) \su K:(K:J) = J$.
 			\item [(iv)]	\emph{ $R$ is almost Gorenstein if and only if $J = (K : \fm) = (x:\fm)$}. By \cite[Proposition 2.3 (4)]{CDKM}, $R$ is almost Gorenstein if and only if $\fm K^2 \su xK$, which is equivalent to say $\fm J \su K$. In this case, as $\fm K = \fm x$ by \cite[Theorem 3.11]{GTP}, we have $\fm J = \fm x$ and so $J\su (x:\fm)\su (K:\fm) \su J$ which means $J = (K:\fm) = (x:\fm)$.
 		\end{itemize}  }
 	\end{rem}

The following result shows that the $\min(S_{{\mathfrak{C}_R}})$ assigns canonical reduction when it equals to two.

 \begin{prop}\label{Min(S_c)=2}
	Let $(R, \fm)$ be a non-regular one-dimensional Cohen-Macaulay local ring admitting a canonical ideal.
	\begin{itemize}		
		\item [(a)] \emph{(See \cite[Lemma 3.6]{CHV})} If $K$ is a canonical ideal of $R$ then $\fm K = \fm(K : \fm)$.		
		\item [(b)] Let $\min( S_{\mathfrak{C}_R})= 2$. Then $R$ has a canonical reduction.
		\item [(c)]  Let $\rho(R) = 2$ and $\min( S_{\mathfrak{C}_R}) \leq 3$. Then $R$ is almost Gorenstein.	
	\end{itemize}
\end{prop}
\begin{proof}
	Let $I$ be a height one ideal of $R$. From the natural exact sequence $$0 \lo \fm(I : \fm)/\fm I \lo I/\fm I \lo (I : \fm)/\fm(I : \fm) \lo (I : \fm)/I \lo 0,$$ it follows that $\fm I = \fm(I : \fm)$ if and only if $\mu(I : \fm) = \mu(I) + \redu(R/I)$. 
	
	(a). By Lemma \ref{mimu}, $\mu(K : \fm) = \redu(\fm)$. Choose an $R$-regular element $x \in \fm$. Then 	$$\redu(\fm) =\ell(\Ext_R^1(R/\fm , \fm))= \ell(\Hom_R(R/\fm , \fm/x\fm)) = \ell((x\fm : \fm)/x\fm).$$ As $R$ is not regular, we have $(x\fm : \fm) = (x : \fm)$ and thus $\fm(x : \fm) = x\fm$. Therefore $\redu(\fm) = \ell((x\fm : \fm)/x\fm) =  \ell((x : \fm)/x\fm) = \mu(x : \fm)$. Note that $\redu(R)=\redu(R/xR)=\ell(\Hom_R(R/\fm, R/xR))=\mu((x:\fm)/xR)=\mu(x:\fm)-1$. Therefore, by the above discussion, we have $\fm K=\fm(K:\fm)$.	
	
	(b). Assume that $K$ is  a canonical ideal of $R$ with $\ell(R/K) = 2$. It follows that $\fm = (K : \fm)$.	Therefore  by (a), we have $\fm^2 = \fm(K : \fm) = \fm K$ and so $K$ is a reduction of $\fm$.
	
	(c).  Assume that $K \in \mathfrak{C}_R$ with $\ell(R/K) = \min( S_{\mathfrak{C}_R})$. By Remark \ref{remrho=2} (iii), $K^2 = xJ$, where $x\in K$ is a reduction of $K$, and $(K : \fm) \su J \su \fm$. Therefore $J = (K:\fm)$ in both case and $R$ is almost Gorenstein by Remark \ref{remrho=2} (iv). 
\end{proof}

 By Proposition \ref{Min(S_c)=2} (a), canonical ideals of a one-dimensional non-regular Cohen-Macaulay local rings cannot be integrally closed. Next  we  deal with the same situation when $\dim R > 1$. First we prove a lemma
  
l \begin{lem}\label{LemInt}
 	Assume that $(R,\fm)$ is a $d$-dimensional non-Gorenstein  Cohen-Macaulay local ring, admitting a canonical ideal $K$. The following conditions are equivalent.
 	\begin{itemize}
 		\item [(i)] $K$ is  integrally closed
 		\item [(ii)]  $R/K$ is reduced.
 	\end{itemize}
 	When this is the case, $R_P$ is regular for all $P \in \emph\Ass_R(R/K)$.
 \end{lem}
 \begin{proof}
 	(i)$\Rightarrow$(ii). let $P \in \Ass(R/K)$. If $R_P$ is not regular then $K_PPR_P = PR_P(K_P : PR_P)$ by Proposition \ref{Min(S_c)=2}, which means $\overline{K_P} \su \overline{(K_P : PR_P)}$. As $K_P$ is integrally closed, this contradiction yields that $R_P$ is regular for all $P \in \Ass(R/K)$ and so $K_P = PR_P$. Therefore, if $P \in \Ass(R/K)$, $(R/K)_P = R_P/PR_P$ and $R/K$ is reduced.
 	
 	(ii)$\Rightarrow$(i). By assumption $K_P = PR_P$ for all $P \in \Ass(R/K)$ and so $K$ is equal to intersection of all associated primes of $R/K$, which means $K$ is integrally closed. 
 \end{proof} 


\begin{cor}\label{CorInt}
	Assume that $(R,\fm)$ is a $d$-dimensional Cohen-Macaulay local ring, with infinite residue field. admitting a canonical ideal. 
	\begin{itemize}
		\item [(a)] Let $K \in \mathfrak{C}_R$ be an integrally closed canonical ideal of $R$. Then, for all canonical ideals $J$ with $K \su J$, $J$ is integrally closed.
		\item [(b)] If all canonical ideals of $R$ be integrally closed, then $R$ is reduced and integrally closed in its total rings of fractions.
	\end{itemize}
\end{cor}
\begin{proof}
	(a). By Proposition \ref{MaxCanIdeal}(i), $K = xJ$ where $xR = (K : J)$. Therefore, as $K$ is integrally closed, $x$ and $J = (K : x)$ are integrally closed.
	
	(b). We show that every principal ideal of $R$, generated by a non-zero divisor, is integrally closed and the result will follow by \cite[Proposition 1.5.2 ]{SH}. Let $x \in \fm$ be a non-zero divisor and $J$ an arbitrary canonical ideal of $R$. Set $K = xJ$. Then by assumption $K$ is integrally closed and, by Proposition \ref{MaxCanIdeal}(i), $xR = (K : J)$ is an integrally closed ideal.
\end{proof} 

 Next proposition shows that, a ring with minimal multiplicity admits a canonical reduction if and only if it is almost Gorenstein.   Recall that $R$ is said to have  {\it minimal multiplicity} if $\mu(\fm) = \e_{\fm}^0(R)$. When $R/\fm$ is infinite, this is equivalent to saying $\fm^2 = x\fm$ for some reduction element $x\in \fm$ (e. g. see \cite[Exercise 4.6.14]{BH}). In this case, $R$ has minimal multiplicity if and only if $\e_{\fm}^0(R) = \redu(R) + 1$.

	 \begin{prop}\label{R/K=2MinMult}
	  Assume that $(R, \fm)$ is a one-dimensional Cohen-Macaulay local ring with infinite residue field and that $K$ is a canonical reduction of $R$. The following statements hold true.
	 	\begin{itemize} 
	 		\item[(a)]  $R$ has minimal multiplicity if and only if $R$ is almost Gorenstein and $\ell(R/K) = 2$.
	 	  \item[(b)] If $\rho(R) = 2$, then $R$ is almost Gorenstein with minimal multiplicity if and only if $\ell(R/K) = 2$. 	
	 	\end{itemize}
	 \end{prop}
	 \begin{proof}
	 	(a). Choose $x \in K$ as a reduction of $\fm$ and set $\e:= \e_{\fm}^0(R)$ and $r = \redu(R)$. Assume that $\e_{\fm}^0(R) = \redu(R) + 1$.	We have, by Proposition \ref{CanRedMaxIdeal}, $\ell(R/K) \leq \e - r + 1 = r + 1 - r + 1 = 2$. As $\ell(R/K) > 1$, $\ell(R/K) = 2$. Therefore $r + 1 = \e = \ell(R/K) + \ell(K/xR) = 2 + \ell(K/xR)$ 	
	 	which means 
	 	$\ell(K/xR) = r - 1$ and $R$ is almost Gorenstein.
	 	
	  The converse is clear as	$\e = \ell(R/xR) = \ell(R/K) + \ell(K/xR) = 2 + r - 1 = r + 1.$
	 	
	(b). We may write $K^2 = xI$ for some ideal $I$. We have $K \subseteq I \subseteq \fm$. Since $R$ is not Gorenstein and $\ell(\fm/K) = 1$, we get $K^2 = x\fm$. Therefore 
	 	$\fm K^2 = x\fm^2 \subseteq xK$ which means $R$ is almost Gorenstein.
	 \end{proof}

  By Proposition \ref{R/K=2MinMult}, if $R$ is a non-almost Gorenstein ring with minimal multiplicity, then $R$ does not have a canonical reduction. To see some examples of such rings one may consult with \cite[Example 3.2 (2), Example 3.4, Example 3.5, Example 5.9]{CDKM}. 
  
 Recall that for an $R$-module $M$, it's trace, denote by $\tr_R(M)$, define as the sum of ideals $f(M)$ where $f \in \Hom_R(M , R)$. Let $R$ admits a canonical module $K_R$. In \cite{HTS}, authors considered  $\tr(K_R)$, the trace of canonical module of $R$, and prove several propositions and define the concept of {\it nearly Gorenstein rings} . A Cohen-Macaulay local ring $R$ called nearly Gorenstein if $\fm \su \tr_R(K_R)$ where $K_R$ denote the canonical module of $R$. As a consequence, \cite[Lemma 2.1]{HTS} yields that, $R$ admits a canonical ideal if and only if $\tr(K_R)$ has an $R$-regular element.
  
 In the next result we  give some characterizations for a one dimensional Cohen-Macaulay local ring to have a canonical reduction. 

  \begin{prop}\label{LemGGL}
		Assume that $(R,\fm)$ is a non-Gorenstein one-dimensional Cohen-Macaulay local ring with infinite residue field which admits a canonical ideal. The following conditions are equivalent.
			\begin{itemize}
				\item [(i)] $R$ has a canonical reduction.
				
				\item [(ii)] $R$ has a canonical ideal $K$  such that, for some regular element $a \in K$, $(a : K)$ is a reduction of $\fm$.
				
				\item [(iii)] Any canonical ideal $I$ of $R$ admits a regular element $y \in I$ such that $(y : I)$ is a reduction of $\fm$. 
				
				\item [(iv)] $\tr(K_R)$ is a reduction of $\fm$.				
			\end{itemize}				
	\end{prop}
	\begin{proof}
(i)$\Rightarrow$(ii). Assume that  $K$ is a canonical ideal of $R$ which is a reduction of $\fm$. As $R/\fm$ is infinite, there exists $x \in K$  which is a reduction of $\fm$. As $xR \subseteq (x : K)$, $(x : K)$ is also a reduction of $\fm$.
		
		(ii)$\Rightarrow$(i). By assumption, there exists $x \in (a : K)$ which is a reduction of $\fm$. As $xK \su aR$, one may write $xK = aJ$ for ideal $J$ of $R$. Therefore  $J$  is  a canonical ideal of $R$. As $J\ne R$,  $J$  is  a reduction of $\fm$.	 
		
		(ii)$\Leftrightarrow$(iii).  Clear as all canonical ideals are $R$-isomorphic.
			
			(ii)$\Leftrightarrow$(iv). (ii)$\Rightarrow$(iv) is clear as $K \su \tr(K_R)$. For converse, let $K$ be an arbitrary canonical ideal of $R$ and set $x \in K$ as a reduction of $K$ and assume that $K^{t+1} = xK^{t}$. By \cite[Proposition 3.3 (3)]{GGHHV}, we have $x\tr(K_R) = K(x:K)$. Multiplying with $K^t$ we have 
			\begin{center}
				$K^{t+1}(x:K) = xK^t\tr(K_R) = K^{t+1}\tr(K_R)$.
			\end{center}  
			which means $(x:K)$ is a reduction of $\tr(K_R)$. Now the result follows by (iii). 
	\end{proof}
	
	Let $R$ be  a one-dimensional Cohen-Macaulay ring admitting a canonical ideal. Then $R$ is generalized Gorenstein if there exists an exact sequence $0 \lo R \lo K \lo C \lo 0$, with $K \in \mathfrak{C}_R$, such that $C$ is free as $R/\fa$-module where $\fa = \Ann(C)$ (see \cite[Theorem 2.3]{GGL}).  
	
	\begin{cor}\label{CRNGGG}
		Assume that $(R,\fm)$ is a one-dimensional Cohen-Macaulay local ring, admitting a canonical ideal $K$.
	\begin{itemize}
		\item [(a)] If $R$ is nearly Gorenstein, or $R$ has self-dual maximal ideal, then $R$ has a canonical reduction.
		\item [(b)] Let $R$ be a generalized Gorenstein ring with respect to $\fa$. Then, $R$ has a canonical reduction if and only if $\fa$ is a reduction of $\fm$.
	\end{itemize}
	\end{cor}
	\begin{proof}
		(a). When $R$ is nearly Gorenstein, $\tr_R(K) = \fm$ and so $R$ has a canonical reduction by Proposition \ref{LemGGL} (a)(iv). The second claim follows from \cite[Proposition 3.6]{K}.
		
		(b). Follows immediately by definition and Proposition \ref{LemGGL}.  
	\end{proof}

		\begin{nota}\label{NSRem}\emph{
			Assume that $k$ is an infinite field. Let $a_1, \ldots , a_n \in \mathbb{Z}$ such that $0< a_1 < a_2 < \cdots < a_n$ and let $\Gamma = \langle a_1, a_2, \ldots , a_n \rangle$ be the numerical semigroup generated by $a_1, \ldots , a_n$. The subring $R := k[[ t^{a_1}, t^{a_2}, \ldots , t^{a_n}]] $ of the power series ring $k[[t]]$ is called  the semigroup ring of $\Gamma$ over $k$. The ring $R$ is a one-dimensional Cohen-Macaulay local domain. The maximum element of the set $\mathbb{V}: = \{ \ \alpha\in\mathbb{Z} \ | \ \alpha\notin\Gamma \ \}$, denoted by $f$, is called  the {\it Frobenius number} of $\Gamma$. The fractional ideal $T = \langle t^{-\alpha} | \ \alpha\in\mathbb{V} \ \rangle$ is a canonical module of $R$ (e.g. see \cite[Example 2.1.9]{GW} ). Next theorem characterizes numerical semigroup rings to have canonical reduction.  
			Note that $R$ is Gorenstein if and only if $\Gamma$ is symmetric, i.e. $\alpha \in \mathbb{V}$ if and only if $f - \alpha \in \Gamma$. }
	\end{nota} 
	
	\begin{thm}\label{NCR} Assume that $R$ is a semigroup ring as in the Notation \ref{NSRem}. Then $R$ has a canonical reduction if and only if $ a_1 +f - \alpha \in \Gamma$ for all $\alpha \in \mathbb{V}$. 
	\end{thm}
	\begin{proof}
		For `` if " part, note that by assumption $ t^{f+a_1}T \su R$ and $t^{a_1} \in t^{f+a_1}T$. Therefore $t^{f+a_1}T$  is a canonical reduction of $\fm$. 
		
		For `` only if "  part,  assume that $I$ is a canonical reduction of $R$. 
		By Proposition \ref{MaxCanIdeal}, $xI = yT$ for some $x , y \in R$. Let $b \in I$ such that $xb = yt^{-f}$, then $y = xbt^{-f}$ and therefore $I = bt^{f}T$. As $R \su t^fT \su \overline{R}$, by \cite[Corollary 2.8]{GTP}, $b \in I$ is a reduction of $I$ . Since $I$ is a canonical reduction, $b$ is also a reduction of $\fm$.  Assume that $b = \overset{h}{\underset{i = 1}{\Sigma}}\lambda_{i}t^{g_i}$, $\lambda_i \in k$ for all $i$, and that $\fm^{s+1} = \fm^s b$. Therefore $t^{a_1(s+1)} \in \fm^s b$ which means $g_j = a_1$ for some $j$, $1\leq j \leq h$, because  $a_1 \leq a_i$ for all $i$, $1\leq i \leq n$. Without loss of generality, we can assume that $a_1 = g_1$ and hence we may write, after multiplying by $\lambda_1^{-1}$, $b = t^{a_1} +  \overset{h}{\underset{i = 2}{\Sigma}}\lambda'_{i}t^{g_i}$, $\lambda'_i \in k$ for all $i$. For any $\alpha \in \mathbb{V}$ we have, by $I = bt^f T$,  $bt^ft^{-\alpha} = t^{f+a_1 - \alpha} + \overset{h}{\underset{i = 2}{\Sigma}}\lambda'_{i}t^{f + g_i - \alpha} \su R$. The last equality means $f+a_1 - \alpha\in\Gamma$ for all $\alpha\in \mathbb{V}$, as desired. 		
	\end{proof}
	\begin{exam}\emph{
			Let $e \geq 4$ and choose $i, j \in \mathbb{N}$ with $0< j-1\leq i<e-j$. The numerical semigroup ring $$R = k[[t^e, t^{e+1}, \ldots , t^{e+i}, t^{e+i+j},t^{e+i+j+1}, \ldots , t^{2e+i+j-1}]]$$
			admits a canonical reduction by Theorem \ref{NCR}. The ideal $(t^{e+1}, t^{e+2},  \ldots , t^{e+j-1})$ is a canonical reduction of $R$. }
	\end{exam}
\begin{exam}\emph{
	Let $a \geq 3$ and $R = k[[t^a, t^{a+1}, t^{a+2}]]$. Then $f = 2a - 1$ and for all $i$, $a+3 \leq i \leq 2a-1$, $t^{i}\notin R$. In this case $R$ has a canonical reduction if and only if $a = 3, 4, 5$ or $6$.}
\end{exam}
	
  We end this section with some results about the relation between canonical reductions and Ulrich ideals. An $\fm$--primary ideal $I$ of $R$ is said to be an {\it Ulrich ideal} if $I^2 = xI$ for some regular element $x \in I$ and $I/I^2$ is a free $R/I$-module (see \cite[Definition 2.1]{GOTWK}). When $I$ is an Ulrich ideal of $R$, one has $(x : I) = I$ \cite[Corollary 2.6]{GOTWK} and the second condition is equivalent to $I/xR$ being $R/I$-free \cite[Lemma 2.3]{GOTWK}. Next theorem gives a characterization, for rings admitting some certain kind of Ulrich ideals, to have a canonical  reduction.  
 
  \begin{prop}\label{URed}
  	Assume that $(R, \fm)$ is a one-dimensional Cohen-Macaulay local ring with infinite residue field such that $R$ admits a canonical ideal. Let $I$ be an $\fm$--primary Ulrich ideal with $\mu(I) > 2$. The following conditions are equivalent.
  	\begin{itemize}
  		\item [(a)]  $R$ has a canonical reduction.
  		
  		\item [(b)] $R$ is generalized Gorenstein local  with respect to $I$, and $I$ is a reduction of $\fm$.
  	\end{itemize} 
  	When this is the case, $\rho(R) = 2$ and $I$ is the unique Ulrich ideal of $R$, which is isomorphic to $(x : K)$ for an arbitrary $K \in \mathfrak{C}_R$ and $x \in K$ as a reduction.
  \end{prop}    
  \begin{proof} First note that if $R$ is generalized Gorenstein then $\rho(R) = 2$ (see \cite[Fact 2.2 and Theorem 2.3]{GGL}). Thus, as mentioned in Remark \ref{remrho=2}, the ideal $(x:K)$ is independent from a particular canonical ideal $K$ and a reduction element $x\in K$.  

  	(a)$\Rightarrow$(b). Let $K$ be a canonical reduction of $R$, $x \in K$ as a reduction and set $\fa := (x : K)$.  First note that if $a \in K$ be a regular element, then $(a : K) \su I$ by \cite[Corollary 2.12]{GTTU}. Therefore, as $K$ can generate by regular elements, we have $K \su I$. Hence $x \in K$ is also  a reduction of $I$, which means $I^2=xI$ and we have 
  	$I \su (x:I) \su (x:K) \su I$.
  	Therefore $K \su I = (x:K)$ and so $\rho(R) = 2$ and $I = \fa = (x:K)$ is the unique Ulrich ideal of $R$. It remains to show that $R$ is generalized Gorenstein.
  	
  	As $\rho(R) = 2$, \cite[Proposition 2.5]{CDKM} (2) and \cite[Theorem 2.3]{GGL} implies that $R$ is generalized Gorenstein if and only if $K^2/x^2R$ is free $R/\fa$--module.
  	Note that $K^2 \su xR$. By Remark \ref{remrho=2} (i), $K^2 = xJ$ for some proper ideal $J$ of $R$ where $K : (x : K) = J$. As $\fa$ is Ulrich ideal, $\fa^2 \su K$ and so $\fa = ( x : K) \su J$, which means $K^2 = x(x : K) = x\fa = \fa^2$. Now, as $\fa/xR$ is free $R/\fa$--module,  we have $K^2/x^2R = x\fa/x^2R \cong  \fa/xR$ and the result follows by \cite[Lemma 2.3]{GOTWK}.  
  	
  	(b)$\Rightarrow$(a). By assumption  $I = ( a : T)$ for some $T \in \mathfrak{C}_R$ and $a \in T$ as a reduction. Choosing $x \in I$ as a reduction element, there exists a canonical ideal $K \in \mathfrak{C}_R$ such that $xT = aK$ and so $x \in K$. As $x$ is a reduction of $I$ and $I$ is a reduction of $\fm$, it follows that $K$ is a reduction of $\fm$ and $R$ has a canonical reduction.	 	
  \end{proof} 
  
 A ring $R$ that admits an Ulrich ideal which is a reduction of $\fm$ may not possesses a canonical reduction (see \cite[Example 5.7 (3)]{GIK}).
  
  Ulrich ideals with minimum number of generators $2$, are less well known. In next section, it will be  shown that if $R$ admits a canonical reduction, then so does $R\ltimes R$, the idealization of $R$ over $R$. In particular, if $a$ is an $R$-regular element, then $aR\ltimes R$ is an Ulrich ideal of $R\ltimes R$ and $\mu(aR\ltimes R) = 2$. 
  
  The notion of 2-$\AGL$ rings is defined in \cite{CDKM} as a partial generalization for almost Gorenstein rings. If $(R, \fm)$ is a one-dimensional Cohen-Macaulay local ring admitting a canonical ideal $K$, then $R$ is 2-$\AGL$ if and only if $\rho(R) = 2$ and $\ell(K^2/xK) = 2$ for some element $x\in K$ as a reduction (see \cite[Lemma 3.1]{CDKM}).  Goto et.al studied in  \cite{GT} the 2-$\AGL$ rings which admit  Ulrich ideals $I$ with $\mu(I)=2$.
  In this connection, we state the following result.
  \begin{cor}
  	Assume that $(R, \fm)$ is a one-dimensional 2-$\AGL$ ring with infinite residue field.  If $R$ admits an  Ulrich ideal $I$ with $\mu(I) = 2$ then $R$ has a canonical reduction.	 
  \end{cor}
  \begin{proof} Let $J$ be a canonical ideal of $R$. Set $x \in J$ as a reduction of $J$ and $\fa = (x : J)$. As $R$ is a 2-$\AGL$ ring, $\ell(R/\fa) = 2$ by \cite[Proposition 3.3]{CDKM}.  By \cite[Proposition 2.3 (1)]{GT}, $J/xR$ is free over $R/\fa$ which means $R$ is a generalized Gorenstein local ring with respect to $\fa$. Therefore, by Proposition \ref{LemGGL}, it is enough to prove that $\fa$ is a reduction of $\fm$. By \cite[Proposition 2.3]{GT}, $\mu(\fa) = \mu(\fm) - 1$ while $\ell(R/\fa) = 2$. Hence the exact sequence $0 \lo \fa/\fm\fa \lo \fm/\fm\fa \lo \fm/\fa \lo 0$ implies that $\ell(\fm/\fm\fa) = \mu(\fm)=\ell(\fm/\fm^2)$, which implies that $\fm^2 = \fa\fm$, that is $\fa$ is a reduction of $\fm$.  		
  \end{proof}


	 \section{Canonical reductions via idealization}
	 
	 The aim of this chapter is to find a characterization for rings with canonical reduction via idealization. 	 
	Throughout this section $(R, \fm)$ is a Cohen-Macaulay local ring with $\dim R = 1$, $M$ an $R$-module. By Nagata extension $R\ltimes M$, we mean idealization of $M$ over $R$. Note that $R\ltimes M$ is a local ring whose maximal ideal is $\fm\ltimes M$.  The ring $R\ltimes M$ is Cohen-Macaulay provided $M$ is a maximal Cohen-Macaulay $R$-module. Also one has $(R\ltimes M)/(\fm\ltimes M) \cong R/\fm$, i.e. $R$ and $R\ltimes M$ has the same residue field. 
	
	Assume that $K_R$ is a canonical module of $R$. In \cite[Proposition 6.1]{GTP}, it is shown that if $I$ is an $\fm$--primary ideal of $R$, then $R\ltimes \Hom_R(I , K_R)$ admits a canonical ideal of the form $I\times L$.  In \cite[Theorem 6.3]{GTP}, a partial converse is provided i.e. if $R\ltimes M$ is almost Gorenstein then $M \cong \Hom_R(I , K_R)$ for some primary ideal $I$ with some extra conditions. 
		
		The next lemma shows that if $R\ltimes M$ has a canonical ideal $I\times L$, then $R$ admits a canonical module $K_R$ and $M \cong \Hom_R(I , K_R)$.
	 
	 \begin{lem}\label{ILemma}
	 	Assume that $(R, \fm)$ is  a one-dimensional Cohen-Macaulay local ring and that $M$ is a maximal Cohen-Macaulay $R$--module. The following statements are equivalent.
	 	\begin{itemize}
	 		\item [(i)] $R\ltimes M$ admits a canonical ideal $I\times L$.
	 		\item [(ii)]  $M$ admits a submodule $L$ which is a canonical module of $R$  such that $\emph\Hom_R(M/L , \E) \cong R/I$ for some ideal $I$  of $R$ of height one.
	 		\item[(iii)] $R$ admits a canonical module $K_R$ and $M \cong \emph\Hom_R(I , K_R)$ for some ideal $I$ of height one.
	 	\end{itemize}
	 
	 \end{lem}
	 \begin{proof}
	 		(i)$\Rightarrow$(ii). As $(R\ltimes M)/(I\times L) \cong (R/I)\ltimes (M/L)$ is a Gorenstein ring, $M/L$ is the canonical module of the zero dimensional ring $R/I$ and so $\Hom_R(M/L , \E) \cong R/I$. Let $x \in \fm$ be a regular element of $R$ so that $(I/xI)\times (L/xL)$ is also the canonical module of $(R/xR)\ltimes (M/xM)$. Let $\bar{t} \in \Soc(L/xL)$ for $t\in L$. Then, for every $(r, w) \in \fm\times M$, we have $(r, w)(0 , \bar{t}) = (0 , r\bar{t}) = (0 , 0)$ which means $0\times\Soc(L/xL) \su \Soc((I/xI)\times (L/xL))$. As $R/\fm \cong (R\ltimes M)/(\fm\times M)$, we have $$\vdim_{R/\fm}(\Soc(L/xL)) \leq \vdim_{R/\fm}(\Soc((I/xI)\times(L/xL)) = 1$$ which means $\redu(L/xL) = 1$ and $L$ is a canonical ideal of $R$. 
	 	
	 	(ii)$\Rightarrow$(i).  As $\Hom_R(R/I , \E) \cong  M/L$, $M/L$ is a canonical ideal of $R/I$ and therefore $R/I\ltimes M/L \cong (R\ltimes M)/(I\times L)$ is Gorenstein.  As $I\times L$ is maximal Cohen-Macaulay $R\ltimes M$--module, we need to prove $\redu(I\times L) = 1$. Choose an $R$--regular element $x \in \fm$. Then $(x , 0)$, the image of $x$ in $R\ltimes M$, is also an $R\ltimes M$--regular and $(I\times L)/x(I\times L) \cong (I/xI)\times(L/xL)$. Hence it is enough to compute $\redu((I/xI)\times(L/xL))$. Note that   $\redu((I/xI)\times(L/xL))= \vdim_{R/\fm}(\Soc((I/xI)\times(L/xL)))$. 
	 	
	 	If $(\bar{a}, \bar{b}) \in \Soc((I/xI)\times(L/xL))$ then, for each $r \in \fm$, $(r\bar{a} , r\bar{b}) = (0 , 0)$ which means $\bar{a} \in \Soc(I/xI)$ and $\bar{b} \in \Soc(L/xL)$. Assume that $\bar{a} \neq 0$. As $L/xL$ is a canonical module of $\bar{R}:=R/xR$, $\Hom_R(\bar{R}/a\bar{R} , L/xL) \neq 0$ and so there exists $\bar{l} \in L/xL$ such that $\bar{a}\bar{l} \neq 0$. Hence $(\bar{a} , \bar{b})(0 , l) = (0 , \bar{a}l) \neq (0 , 0)$ which is a contradiction because $\fm\times M$ is the maximal ideal of $R\ltimes M$ and $(\bar{a}, \bar{b}) \in \Soc((I/xI)\times(L/xL))$. Therefore $\Soc((I/xI)\times(L/xL))\cong 0\times\Soc (L/xL)$.  As $L$ is canonical module of $R$, we find that $\redu((I/xI)\times(L/xL))=\redu(L/xL)=1.$ Therefore $I\times L$ is a canonical ideal of $R\ltimes M$.
	 	
	 	(ii)$\Rightarrow$(iii). Applying $\Hom_R( - , L)$  gives the exact sequence
	 	$$ 0 \lo \Hom_R(M , L) \overset{f}{\lo} \Hom_R(L , L) \lo \Ext_R^1(M/L , L) \lo 0.$$
	 	
As $\Hom_R(L , L){\cong} R$, there exists an ideal $J$ of $R$ such that $J\cong\Hom_R(M , L)$ and $R/J \cong \Ext_R^1(M/L , L)$. Our assumption, $\Hom_R(M/L , \E) \cong R/I$, implies that $IM \su L$. If $x\in I$ is an $R$--regular element, $\Ext_R^1(M/L , L) \cong \Hom_R(M/L , L/xL) \cong \Hom_R(\Hom_R(R/I , \E) , L/xL)$.  Therefore $$J=\Ann_R(\Ext_R^1(M/L , L) ) =\Ann_R(\Hom_R(\Hom_R(R/I , \E) , L/xL))=I$$ and the result follows.

	(iii)$\Rightarrow$(ii). Follows by the result of Goto-Matsuoka-Phuong, \cite[Proposition 6.1]{GTP}.
	 \end{proof}
   
	   Lemma \ref{ILemma} shows that, in order to study idealizations of $R$ that admits canonical ideals, it is enough to consider idealizations of the form $R\ltimes \Hom_R(I , K_R)$, where $I$ is a height one ideal of $R$ and $K_R$ the canonical module of $R$. 
	   
	   The next proposition shows that when an idealization of $R$ has a canonical reduction. 
	   
	   \begin{prop}\label{IProp1}
	   Assume that $(R, \fm)$ is a one-dimensional Cohen-Macaulay local ring and $M$ is a maximal Cohen-Macaulay $R$--module and that $I$ is an ideal of $R$. The following statements are equivalent.
	   \begin{itemize}
	   	\item [(i)] $R\ltimes M$ has a canonical reduction $I \times L$ for some submodule $L$ of $M$.
	   	\item [(ii)] $M \cong \emph\Hom_R(I , K_R)$ and $I$ is a reduction $\fm$.
	   \end{itemize}  
	   \end{prop}  
	\begin{proof}
	(i)$\Rightarrow$(ii). By Lemma \ref{ILemma}, $M \cong \Hom_R(I , K_R)$. Since $I\times L$ is  a reduction for $\fm\times M$, there exists a positive integer $t$ such that $(\fm\times M)^{t+1} = (I\times L)(\fm\times M)^t$. Note that $(\fm\times M)^{t+1} = \fm^{t+1}\times \fm^tM$ and $(I\times L)(\fm\times M)^t = I\fm^t\times (\fm^tL + I\fm^{t-1}M)$. Comparing the first summands, gives $\fm^{t+1} = I\fm^t$ that is $I$ is a reduction of $\fm$.
	
	(ii)$\Rightarrow$(i).  By \cite[Proposition 6.1]{GTP}, there exists a submodule $L$, isomorph to $K_R$, of $M$ such that $I\times L$ is a canonical ideal of $R$. Assume that $\fm^{s} = I\fm^{s-1}$. It follows that $(I\times L)(\fm\times M)^s= (\fm\times M)^{s+1}$.
	\end{proof}

	 In \cite[Theorem 6.5]{GTP} it is proved that, for a one-dimensional Cohen-Macaulay local ring $(R, \fm)$, $R\ltimes\fm$ is almost Gorenstein if and only if $R$ is almost Gorenstein. Next theorem shows that the same happens for rings with canonical reduction. Moreover, it shows that the class of rings $R$ with canonical reductions is a natural notion when we compare it with $R\ltimes R$.
	 
	 \begin{thm}\label{IdealRed}
	 	Assume that $(R, \fm)$ is a one-dimensional Cohen-Macaulay local ring. The following conditions are equivalent.
	 	\begin{itemize}
	 		\item [(i)] $R$ has a canonical reduction $K$.
	 		\item [(ii)] $R\ltimes I$ has a canonical reduction for any ideal $I$ containing $K$.
	 		\item [(iii)] $R\ltimes\fm$ has a canonical reduction.
	 		\item[(iv)] $R\ltimes R$ has a canonical reduction.
	 	\end{itemize}
	 \end{thm}
	 \begin{proof}
	 	(i)$\Rightarrow$(ii). Let $I$ be an ideal of $R$ containing $K$, $j$ the inclusion map from $K$ to $I$. Consider the $R$-monomorphism $\Hom_R(I , K) \overset{j^*}{\lo} \Hom_R(K , K)$. As there is an $R$-isomorphism $\Hom_R(K, K)\overset{\psi}{\lo} R$, each element $f \in \Hom_R(I , K)$ corresponds to some $ r\in R$ depending on $f$. 
	 	
	 	On the other hand, for each $a \in K$, consider $f_a \in \Hom_R(I , K)$ such that $f(x) = xa$ for all $x \in I$. Therefore $j^*(f_a)$ corresponds to  $a$. In other words, $\psi \circ j^*(\Hom_R(I , K))$ contains $K$. Note that the ideal $T:=\psi \circ j^*(\Hom_R(I , K))$ is a proper ideal since $K$ is a maximal canonical ideal by Proposition \ref{CanRedMaxIdeal}. As $K$ is a reduction of $\fm$, $T$ is also a reduction  of $\fm$. Now, by Proposition \ref{IProp1}, $R\ltimes\Hom_R(\Hom_R(I , K) , K) \cong R\ltimes I$ has a canonical reduction.
	 	
	 	(ii)$\Rightarrow$(iii). Clear since $K \su \fm$.
	 	
	 	(iii)$\Rightarrow$(i). As $R\ltimes\fm$ has a canonical reduction, by Proposition \ref{IProp1}, $R$ admits the canonical module $K_R$ and there exists an ideal $I$ of $R$ of height one such that  $\fm \cong \Hom_R(I , K_R)$. Therefore, $I \cong \Hom_R(\fm , K_R)$ and $I$ is a reduction for $\fm$.  As, by \cite[Corollary 2.4]{GHI}, $\fm$ is not a canonical ideal of $R$, we have $I \ncong R$. Lemma \ref{ILemma} suggests that a canonical reduction of $R\ltimes \fm$ is of the form $I\times L$ for some proper ideal $L$ such that $L$ is a canonical ideal of $R$. Moreover $R/I \cong \Hom_R(\fm/L , \E)$. 
	 	
	 	 From the exact sequence $0 \lo \fm/L \lo R/L \lo R/\fm \lo 0$ we obtain the exact sequence $$0 \lo \Hom_R(R/\fm , \E) \lo \Hom_R(R/L , \E) \lo \Hom_R(\fm/L , \E) \lo 0.$$ As $L$ is a Gorenstein ideal of $R$ we have $ \Hom_R(R/L , \E)\cong R/L$ and so $\Hom_R(R/\fm , \E) \cong J/L$ for some ideal $J$ of $R$ containing $L$. As a result $\Hom_R(\fm/L , \E) \cong R/J$. Note that $\fm J \su L$, which means $J/L \su (L : \fm)/L \cong R/\fm$ because $R/L$ is Gorenstein. Therefore, $J = (L : \fm)$ and $\Hom_R(\fm/L , \E) \cong R/(L : \fm)$ that is $I = (L : \fm)$. As $L$ is canonical ideal of $R$, Theorem \ref{Min(S_c)=2} (a) implies that $\fm L = \fm(L : \fm)$. Putting together, we find $\fm L=\fm I$. As $I$ is a reduction of $\fm$ so is $L$.
	 	 
	 	 (i)$\Leftrightarrow$(iv). The proof is a straightforward application of Proposition \ref{IProp1}.
	 \end{proof}

	 \begin{rem}\label{Irem} \emph{Here we give a method to construct rings with canonical reduction of arbitrary large canonical index. }
	 	
	  \emph{ Choose $e \geq 4$ and $r \in \{ 2, \ldots , e-1 \}$. Then, by \cite[Example 4.4]{Sally}, there exists a one dimensional Cohen-Macaulay complete local ring $(R, \fm)$ such that $\mu(\fm) = e - 1$, $\e_{\fm}^0(R) = e$ and $\redu(R) = e - 2$, where $\re(\fm) = r$. By Proposition \ref{IProp1}, $R\ltimes\Hom_R(\fm , K_R)$ has a  canonical reduction of the form $\fm\times L$, where $K_R$ denotes the canonical module of $R$. Note that $x\fm^r=\fm^{r+1}$ for some element $x\in\fm$. It is easy to see that $(x , 0)$ is a reduction of $\fm\times L$.
	 		Now, in the following, we study the properties of the ring $A:= R\ltimes\Hom_R(\fm, K_R)$.
	 		 	\begin{itemize}
	 		 		\item[(i)] {\it Canonical index of $A$}. For any integer $t > 0$, one has $$(\fm\times L)^{t+1} = (\fm^{t+1} , \fm^tL)\ \  \text{and} \ (x , 0)(\fm\times L)^t = (x\fm^t , x\fm^{t-1}L).$$ In order to calculate the canonical index of $A$, we compare $(\fm\times L)^{t+1}$ and $(x , 0)(\fm\times L)^t$.  If $\fm^rL = x\fm^{r-1}L$, then $\re(\fm\times L) = r$. Otherwise $\re(\fm\times L) = r + 1$.
	 			\item [(ii)] {\it $A$ is not Gorenstein.} Otherwise, by \cite[Theorem 3.7]{GTP}, we equivalently have $\fm L = xL$ that is $\redu(R) = \e_{\fm}^0(R)$, a contradiction.   
	 			\item [(iii)] {\it $A$ is not generalized Gorenstein for $r>2$.} If $r > 2$, then $\re(\fm\times L) > 2$ by part (i), which means $R\ltimes\Hom_R(\fm , K_R)$ is not a generalized Gorenstein local ring \cite[Proposition 2.1]{GGL}.
	 			\item [(iv)] $\min(S_{{\mathfrak{C}_A}}) = 2$. As $K_A=\fm\times L$ is the canonical reduction of $A$, we have   $$\ell((R\ltimes\Hom_R(\fm , K_R))/(\fm\times L) = \ell((R/\fm)\ltimes(\Hom_R(\fm , K_R)/L) = 2.$$
	 		\end{itemize}   }
	 \end{rem}

			\section{Rings with canonical reductions of higher dimension}
			
			In this section, for Cohen-Macaulay local rings of higher dimensions, we define and study canonical reductions.
		
		\begin{defn}\label{CanRedHigherDim}\emph{Assume that $(R,\fm)$ is a $d$-dimensional Cohen-Macaulay local ring admitting a canonical ideal. A canonical ideal $K$ is said to be a   \emph{canonical reduction} of $R$ if there exists an equimultiple ideal $I$, with $\h(I) = d-1$ such that $K + I$ is a reduction of $\fm$.}
		\end{defn}
			
	This definition coincides with Definition \ref{DCR1} when $\dim R=1$. Note that a canonical reduction, if exists, is not necessarily equimultiple when $d > 1$.
		
	\begin{rem}\label{CanRednot}\emph{Let $R$ be a $d$-dimensional Cohen-Macaulay local ring with infinite residue field which admits a canonical ideal $K$.
	\begin{itemize}
		\item  [(a)] If $K$ admits a superficial element for $\fm$, then $R$ has a canonical reduction in the sense of Definition \ref{CanRedHigherDim} since any superficial element can be extended to a superficial sequence which is a minimal reduction of $\fm$.
		\item  [(b)] If $K$ is a canonical reduction of $R$ and $I$ is an equimultiple ideal of height $d-1$ such that $K+I$ is a reduction of $\fm$, then we may choose $\fq = (x_2, \cdots , x_d)$ as a minimal reduction of $I$ such that $\fq$ is a parameter ideal of $R$, and $\fq(R/K)$ is a minimal reduction of $\fm/K$. Therefore $\fq(R/K)$ is a parameter ideal for $R/K$ and we have $K(R/\fq) = (K + \fq)/\fq \cong K/\fq\cap K = K/\fq K$, so $K(R/\fq)$ is a canonical reduction of $R/\fq$.
	\item [(c)] Let $K$ be a canonical reduction of $R$ so that $I$ is an equimultiple ideal of height $d-1$ such that $K+I$ is a reduction of $\fm$. Assume further that $K$ is equimultiple and $x\in K$ is a reduction of $K$. If $\fq$ is chosen as in (b), then $xR + \fq$ is a minimal reduction of $\fm$.
						\end{itemize}
						 }		
				\end{rem}

	Next we show that almost Gorenstein local rings and nearly Gorenstein rings of dimension $d > 1$ have a canonical reduction. First we prove the following lemma which is fundamental in the sequel.

			\begin{lem}\label{LemSup}
				Assume that $(R,\fm)$ is a $d$-dimensional Cohen-Macaulay local ring, with infinite residue field, admitting a canonical ideal $K$. Let $y \in \fm$ be a regular element of $R$, and  let $\fq = (x_2, x_3, \cdots , x_d)$ be an ideal generated by a regular sequence for $R$, $R/K$ and $R/yR$.  Let  $x \in K \setminus \fm K$ such that $xR + \fq$ is a parameter of $R$ and set $\bar{R} = R/\fq$. If  $y\bar{R} \in (x \bar{R} :_{\bar{R}} K\bar{R})$ then there exists a canonical ideal $I$ of $R$ such that $y \in I$. 
			\end{lem}
			\begin{proof}
				Note that $K(R/\fq) = (K + \fq)/\fq \cong K/\fq K$ is a  canonical ideal of $R/\fq$ since $x_2, \cdots, x_d$ is regular over $R/K$. 	Assume that $I$ is the ideal of $R$ such that $y\in I$ and  $yK\bar{R} = xI\bar{R}$. We show that $I$ is a canonical  ideal of $R$. As $y$ is regular over $R/\fq$,  $yK\bar{R}$ and so $xI\bar{R}$ are canonical modules for $\bar{R}$. We claim that	$x_2, \cdots, x_d$ is a regular sequence on $R/xI$ so that $R/xI$ is Cohen-Macaulay and $\dim(R/xI) = d - 1$. To see this, assume that $x_2$ is not regular over $R/xI$. Then $x_2 \in \fp$ for some $\fp\in \Ass_R(R/xI)$. Therefore $\fp = (xI : \alpha)$ for some $\alpha \in R$. As
				$x_2 \in \fp$, we have $x_2\alpha \subset xI$. Note that $(x , x_2)$ is a regular sequence on $R$, so we must have $\alpha \in xI$ which is a contradiction. Continuing this way, since $(x , x_2, \ldots , x_i)$ is $R$-sequence for each $i$, $2 \leq i \leq d$, we see that $x_2, x_3, \ldots, x_d$ is a regular sequence for $R/xI$.
				
			Therefore,  we have $I/\fq I \cong xI/\fq xI \cong xI/\fq\bigcap xI \cong (xI + \fq)/\fq = xI(R/\fq) = yK(R/\fq)$ which means $I/fq I$ is a canonical module of $R/\fq$. Also the exact sequence $ 0 \lo xI \lo R \lo R/xI \lo 0$ implies that $xI$ is maximal Cohen-Macaulay and $\dim(xI) = d = \depth(xI)$. Therefore $I$ is a maximal Cohen-Macaulay module while $\fq$ is a regular sequence for $I$, and $I/\fq I$ is a canonical module of $R$. Therefore $I$ is canonical ideal of $R$.			
			\end{proof}

					Let $I$ be arbitrary ideal of $R$ with positive height. An element $x \in R$ is called a {\it reduction element for $I$} if there exist $x_2, \cdots , x_t \in R$ such that $(x, x_2, \cdots , x_t)$ is a minimal reduction of $I$.   Next result is a direct corollary from Proposition \ref{LemGGL} and Lemma \ref{LemSup}.
					
	\begin{prop}\label{CanRedNG}
						Assume that $(R,\fm)$ is a Cohen-Macaulay local ring, admitting a canonical ideal, such that $\emph\tr(K_R)$ admits a reduction element for $\fm$. Then $R$ admits a canonical reduction. In particular, every nearly Gorenstein ring admits a canonical reduction. 
\end{prop}
\begin{proof}
					Let $K$ be an arbitrary canonical ideal of $R$ and assume that $x_1\in \tr_R(K)$ is a superficial element of $\fm$. Choose $(x_2, x_3, \cdots , x_d)$ as a superficial sequence for $R/x_1R$, which is also regular over $R/K$. Set $\fq =  (x_2, x_3, \cdots , x_d)$, $\bar{R} = R/\fq$ and $\bar{K} = K(R/\fq)$. By \cite[Lemma 5.1 (ii)]{HTS}, $\tr_R(K)\bar{R} \su \tr_{\bar{R}}(\bar{K})
					$ which means $\tr_{\bar{R}}(\bar{K})$ admits a superficial element of $\bar{\fm}$ and $\bar{R}$ has a canonical reduction by Proposition \ref{LemGGL} (iv). Again by Proposition \ref{LemGGL} (iii), there exists $\bar{y} \in \bar{K}$ such that $(\bar{y} : \bar{K})$ is a reduction of $\bar{\fm}$. Hence there exists $x \in \fm$ such that $\bar{x} \in (\bar{y} : \bar{K})$ is a minimal reduction for $\bar{\fm}$. By Lemma \ref{LemSup}, there exists a canonical ideal $I$ of $R$ such that $x \in I$. As $(x_2, \cdots , x_d)$ is a superficial sequence for $\fm$, $xR + \fq$, and so $I + \fq$ is a reduction of $\fm$ and $R$ has a canonical reduction.
				\end{proof}
				
				Note that when $R$ admits a canonical reduction then the trace of the canonical module of $R$ may not  admit a reduction element for $\fm$. But this is true if  canonical ideals of $R$ are equimultiple, see Notation \ref{DCR1}. 	Unfortunately, not all rings with canonical reductions have equimultiple canonical ideals, even if they admit a canonical ideal which has a reduction element for maximal ideal (as happens if $R$ is generalized Gorenstein or nearly Gorenstein ring). Next proposition shows a necessary and sufficient condition for a ring to have an equimultiple canonical ideal. Recall that an ideal $I$ is called {\it normally torsion-free} if $\Ass_R(R/I) = \Ass_R(R/I^n)$ for all $n>0$.

				\begin{prop}\label{equi}
					Assume that $(R,\fm)$ is a Cohen-Macaulay local ring with infinite residue field and that $K$ is a Cohen-Macaulay ideal of height one. Then $K$ is equimultiple if and only if $K$ is normally torsion-free.
				\end{prop}
				\begin{proof}
					First let $x \in K$ be a reduction of $K$, which means $K$ is equimultiple. Then $x^n$ is a reduction of $K^n$ for all $n>0$, and so $\Ass_R(R/K^n) = \Ass_R(R/x^nR)$ for all $n>0$. Now the result follows as $\Ass(R/xR) = \Ass(R/x^nR)$ for all $n>0$. Conversely, assume that $K$ is normally torsion-free. Let $\fq = (x_2, \cdots , x_d)$ be  an ideal generated by  a part of a system of  parameters for $R$ and so $\fq(R/K)$ is a parameter ideal for $R/K$. We will show that $x_2, \cdots , x_d$ is a superficial sequence for $K + \fq$. To this, it is enough to show that $x_d(R/\fq')$ is a  superficial element of $(K+\fq)/\fq'$, where $\fq' = (x_2, \cdots , x_{d-1})$ (set $\fq'=0$ if $d=2$). Let $r \in R$ such that $rx_d(R/\fq') \in ((K+\fq)/\fq')^{n+1}$. We will use \cite[Lemma 8.5.3]{SH} and show that $r(R/\fq') \in ((K+\fq)/\fq')^{n}$. As $(K+\fq)^{n+1} =  K^{n+1} + \fq K^n +  \cdots \fq^nK + \fq^{n+1}$ while $\fq = (x_2, \cdots, x_d) = \fq' + x_dR$, after a rearrangement  we have 
					$$rx_d = k_{n+1} + x_dk_{n} + \cdots + x_d^nk_{1} + ax_d^{n+1} +  a'$$ where  $k_{i} \in K^i$ for $1 \leq i \leq n+1$, $a \in R$ and $a' \in \fq'$. Therefore $$x_d(r - k_n - \cdots - x_d^{n-1}k_1 - ax_d^n) = k_{n+1} + a' \in K^{n+1} + \fq'.$$ By assumption $x_d$ is regular over $K^{n+1} + \fq'$, which means $r - k_n - \cdots - x_d^{n-1}k_1 - ax_d^n \in K^{n+1} + \fq'$ and so  $$ r \in K^{n+1} + K^n + \cdots + x_d^{n-1}K + x^nR + \fq' \su (K+\fq)^{n} + \fq'.$$		
				\end{proof}
				
Sometimes it is easier to show that when canonical ideal is not equimultiple. Next  result generalized \cite[Lemma 3.6]{GGHV}.
				
				\begin{prop}\label{Neq}
Assume that $(R,\fm)$ is a Cohen-Macaulay local ring with infinite residue field, admitting a canonical ideal. If $R$ is Gorenstein in codimension one then  canonical ideals of $R$ are not equimultiple. 	
				\end{prop}
				\begin{proof}
Let $K \in \mathfrak{C}_R$ and $x \in K$ be a superficial element of $K$. Assume that $\fp \in \Ass_R(R/K)$ such that $R_{\fp}$ is Gorenstein. Then, by Lemma \ref{suploc}, $K_{\fp} = xR_{\fp}$. Therefore, if we set $C = K/xR$ then $\fp \notin \Ass_R(C)$. This means $\Ass_R(C) = \Ass_R(R/xR) \backslash \{ \fp \in \Ass(R/K) \ | \ R_P \ \text{is Gorenstein} \ \}$. Now consider the exact sequence $0 \lo R \overset{\times x}{\lo} K \lo C \lo 0$. If $R$ is Gorenstein in codimension one, then $C_{\fp}$ for all $\fp \in \Spec R$ with $\h(\fp) = 1$. Hence $C = 0$ and $R$ is Gorenstein.
				\end{proof}

In order to continue our discussion we recall the definition of almost Gorenstein and generalized Gorenstein local rings.
						
	\begin{defn}\label{defgg}\emph{ Assume that $(R,\fm)$ is a Cohen-Macaulay local ring such that there exists an exact sequence 
				\begin{center}
						$0 \lo R \lo K \lo C \lo 0.$	\end{center} 
						Then $R$ is {\it generalized Gorenstein} with respect to an $\fm$-primary ideal $\fa$ if 	\begin{itemize}	\item [(1)] $C/\fa C$ is a free $R/\fa$-module, and 	
						\item [(2)] $\fa C = (x_2, x_3, \cdots , x_d)C$ for some sequence  $x_2, x_3, \cdots , x_d\in \fa$ of  a system of parameters.		\end{itemize}	By (\cite[Definition 3.3]{GTT}) $R$ is called {\it almost Gorenstein}  if $\mu(C) = \e_{\fm}^0(R)$, which means $\fa = \fm$. }	\end{defn}
						
		\begin{prop}\label{AGCR}			Assume that $(R,\fm)$ is a $d$-dimensional generalized Gorenstein local ring with respect to an $\fm$-primary ideal $\fa$, with infinite residue field. Consider the next two conditions.
			\begin{itemize}
			\item [(i)] $R$ has a canonical reduction.
			\item [(ii)] $\fa$ admits a reduction element for $\fm$.
			\end{itemize}
		Then (ii)$\Rightarrow$(i). If $R$ admits an  equimultiple canonical reduction then (i)$\Rightarrow$(ii).  In particular, almost Gorenstein rings have a canonical reduction. 
	\end{prop}
	\begin{proof}
	Since $R$ is generalized Gorenstein, there exists an exact sequence $0 \lo R \overset{\phi}{\lo} K \lo C \lo 0$ such that $K$ is canonical ideal of $R$ and $C$ is as in Definition \ref{defgg}. Set $y = \phi(1)$.
	
	(i)$\Rightarrow$(ii). Let $K$ be a canonical ideal of $R$, which is equimultiple, and $x \in K$ as a reduction. Then, by Notation \ref{CanRednot}(c), $x$ is a part of a reduction of $\fm$ while $x \in \fa$ since $K \su \fa$ by \cite[Theorem 1.2]{GGL}.  
	
	 (ii)$\Rightarrow$(i).  By \cite[Corollary 8.5.9.]{SH} there exists an $R$-regular  element $x_2 \in \fa$ such that $x_2$ is superficial for $C$,  $x_2$ is regular on $R/K$,  on $R/yR$, and on $R/xR$. If $d> 2$ then  one may continue in this way to find an ideal $\fq = (x_2, x_3, \ldots , x_d)$ generated by an $R/xR$-regular sequence such that $\fq$ is a superficial sequence for $C$, $x_2, x_3, \ldots , x_d$ is a system of parameters for $R/K$,  $yR + q$ is a parameter ideal of $R$, and  $xR + q = (x, x_2, \ldots , x_d)$ is a parameter ideal for $R$. Now, \cite[Theorem 2.4]{GGL} implies that $R/q$ is a one-dimensional generalized Gorenstein ring with respect to $\fa/q$. Therefore, $\fa/q = (y(R/q) : K(R/q))$. By Lemma \ref{LemSup}, there exists a canonical ideal $I$ of $R$ such that $x \in I$, and so $R$ has a canonical reduction.
	 	\end{proof}

One may notice that the ring which is discussed in \cite[Example 5.1]{GGL} admits a canonical reduction by Proposition \ref{AGCR}.

 Next remark deals with some special rings with canonical reduction, using linkage theory.
 
 \begin{rem}\emph{
 	Assume that $(R,\fm)$ is $d$-dimensional Cohen-Macaulay local ring, admitting a canonical ideal. If $I$ and $J$ be geometrically linked ideals, i.e. $I = (0:J)$ and $J=(0:I)$ with $\Ass_R(R/I)\cap\Ass_R(R/J) = \emptyset$, then $I + J$ contains a non-zerodivisor. Assume that there exists a regular sequence $\fq$ over $R/(I+J)$, such that $I+J + \fq$ is a reduction of $\fm$ (for example $I+J$ contains a superficial element for $\fm$). Then, as $(I+J)/J$ is a canonical ideal of $R/J$ (respectively $(I+J)/I$ is a canonical ideal of $R/I$) and $\fq(R/J)$ is a parameter for $R/J$ (respectively $\fq(R/I)$ is a parameter for $R/I$) then $R/I$ and $R/J$ both admits a canonical reduction. }
 \end{rem}
   
   Assume that $R$ is a Gorenstein local ring, and $I$ and $J$ are linked and $R/I$ is Cohen-Macaulay ring. If $\mu(J) = 1$ then  $R/I$ is a Gorenstein ring since $K_{R/I} = \Hom_R(R/I , R) = J$. Next proposition generalized this result to rings with canonical reductions.  
   
 \begin{prop}\label{LinkCanRed}
 	Assume that $(R,\fm)$ is a $d$-dimensional Cohen-Macaulay local ring, admitting canonical reduction $K$, and that $I$ and $J$ are ideals of $R$ with $R/I$ is Cohen-Macaulay.  
 	If $I$ is linked to $J$ with respect to $\q$, where $q$ is a maximal regular sequence contained in $I\cap J$, such that $R/q$ has a canonical reduction  and $J/q$ is cyclic then $R/I$ has a canonical reduction.
 
 \end{prop}
 \begin{proof}
 	After all, we may assume that $I$ and $J$ are zero-height ideals such that $J=aR$ for some $a\in \fm$ and $I = (0:a)$ and $aR = (0:I) \cong R/I$. By above discussion, if $R$ is Gorenstein then so is $R/I$ and there is nothing to prove. Hence we may assume that $R$ is not Gorenstein. Let $K_{R/I}$ denotes the canonical module of $R/I$. Note that 
  	\begin{center}
  	$R/I \cong \Hom_R(R/I , R) \cong \Hom_R(R/I , \Hom_R(K , K)) \cong \Hom_R(K/IK , K)$
  	\end{center} 

 which means $K_{R/I} = \Hom_R(R/I , K) \cong K/IK$. Applying $\Hom_R(- , K)$ over the exact sequence $0 \lo R/I \overset{\times a}{\lo} R \lo R/J \lo 0$, with the fact that $\Hom_R(R/J , K) \cong (0:_K J) = I\cap K$, gives the commutative diagram
 $$\begin{CD}
 &&&&&&&&\\
 \ \ 0 @>>> \Hom_R(R/J , K) @>>> \Hom_R(R , K) @>>> \Hom_R(R/I , K) @>>> 0 &  \\
 && @A{g}A{\cong}A @A{g'}A{\cong}A \\
 \ \  0 @>>> I\cap K @>>> K @>>> K/I\cap K @>>> 0 &\\
 &&&&&&  \\
 \end{CD}$$ 
 where $g$ and $g'$ are isomorphisms with
 $g(a)(\bar{1})= a$ for $a \in I\cap K$ and $g'(k)(1) = k$ for $k \in K$. Therefore there is a well-defined homomorphism $\theta : K/I\cap K \lo \Hom_R(R/I , K)$ such that $\theta(\bar{k})(\bar{1}) = k$ for $k \in K$.  If $\theta' : \Hom_R(R/I , K) \lo K/IK$ be the natural homomorphism, which send $f \in \Hom_R(R/I , K)$ to $\overline{f(\bar{1})} \in K/IK$, then there is a well-defined homomorphism $\theta'\theta : K/I\cap J \lo K/IK$ such that $\theta'\theta(k + I\cap K) = k + IK$ for each $k \in K$. Hence $I\cap K = IK$ and $K(R/I)$ is a canonical ideal of $R/I$. Now, if $\fq$ be the parameter ideal of height $d-1$ such that $K + \fq$ is a reduction of $\fm$ then $\fq(R/I)$ is also a parameter ideal for $R/I$ and $K(R/I) + \fq(R/I)$ is a reduction of $\fm/I$   and $R/I$ has a canonical reduction.
  \end{proof}

  Next we investigate transferring canonical reductions by some specific ring homomorphisms. 	First we prove a lemma.
 
 \begin{lem}\label{LemMul}
 	Assume that $\varphi : (R,\fm) \lo (S, \fn)$ is a flat local homomorphism between $d$-dimensional Noetherian local rings. Then $\e_{\fm S}^0(S) = \ell_S(S/\fm S)\e_{\fm}^0(R)$. 
 \end{lem} 	
		\begin{proof}
			 It is an easy exercise to prove that $\ell_S(S\otimes_R M) = \ell_S(S/\fm S)\ell_R(M)$ when $M$ is a finite length $R$-module. Now the result follows by comparing Hilbert functions  $\H_S(\fm S , S)$ and $\H_R(\fm , R)$.
		\end{proof}

		 \begin{thm}\label{canredflat}
			Assume that $\varphi : (R,\fm) \lo (S, \fn)$ is a flat local ring homomorphism between generically Gorenstein Cohen-Macaulay local rings of positive dimension $d$ and $n$, respectively, such that  $S/\fm S$ is Gorenstein. Consider the two following conditions.
				\begin{itemize}
					\item [(a)] $R$ has a canonical reduction and $\e_{\fn}^0(S) = \e_{\fn}^0(S/\fm S)\e_{\fm}^0(R)$.
					\item [(b)] $S$ has  a canonical reduction.
				\end{itemize}
			Then 
			\begin{itemize}
				\item [(i)] Let $n = d$. Then (a)$\Rightarrow$(b). The converse holds if $R/\fm$ is infinite.
				\item [(ii)] Let $n > d$. Then  (a)$\Leftrightarrow$(b)  if $R/\fm$ is infinite and $S/\fm S$ is a regular ring.
			\end{itemize}
				\end{thm}
			\begin{proof}
				(i). First let $d = 1$. If $R$ has a canonical reduction $K$, then $KS$ is a canonical ideal of $S$ and is a reduction of $\fm S$. By assumption, $\e_{\fn}^0(S) = \ell_S(S/\fm S)\e_{\fm}^0(R) = \e_{\fm S}^0(S)$ which means $\fm S$ is a reduction of $\fn$ by \cite[Theorem 11.3.1]{SH}. Therefore $KS$ is a canonical reduction of $S$. Conversely, let $S$ has a canonical reduction. If $K$ is a canonical ideal of $R$ then, by Proposition \ref{LemGGL} (iii), there exists $x \in K$ such that $(xS :_S KS)$ is a reduction of $\fn$. Therefore $(xS :_S KS)$ is a reduction of $\fm S$ and by faithful flatness, $(x :_R K)$ is a reduction of $\fm$ and $R$ has a canonical reduction. Also, as $\fm S$ is a reduction of $\fn$, by Lemma \ref{LemMul} we have $\e_{\fn}^0(S) = \e_{\fm S}^0(S) = \ell_S(S/\fm S)\e_{\fm}^0(R)$.
				
				Now, assume that $d>1$. If $R$ admits a canonical reduction $K$, then there exists a an equimultiple ideal $I$ of length $d-1$ such that $K + I$ is a reduction of $\fm$. It follows that $KS + IS$ is a reduction of $\fm S$. Now the result follows since $IS$ is equimultiple in $S$ and $\fm S$ is a reduction of $\fn$.
				
				 Assume that $S$ admits a canonical reduction $T$ and let  $K\in \mathfrak{C}_R$ with $K \su \fm^2$. Choose a superficial sequence $\fx = x_2, \cdots , x_d$ for $\fm$, such that (1) $\fx(R/K)$ is a regular sequence for $R/K$, (2) $\fx S$ is  $S$-regular and (3) $\fx S(S/T)$ is a regular sequence for $S/T$. This is quite possible since $\varphi$ is faithful flat. Let $\fq = (x_2, \cdots , x_d)$ and set $\bar{R} = R/\fq$ and $\bar{S} = S/\fq S$. Consider the flat local homomorphism $\bar{\varphi} : \bar{R} \lo \bar{S}$. Both of $\bar{R}$ and $\bar{S}$ are one-dimensional local, and $\bar{S}$ has a canonical reduction $T\bar{S}$. Therefore, since $K\bar{S}$ is a canonical ideal of $\bar{S}$, by Proposition \ref{LemGGL} (iii) there exists a regular element $x \in K$ such that $((x\bar{S} :_{\bar{S}} K\bar{S}))$ is a reduction of $\fn/\fq S$. Hence $(x(R/\fq) : K(R/\fq))$ is a reduction of $\fm/\fq$ which means $R/\fq$ admits a canonical reduction. Let $x_1 \in \fm$ such that $x_1(R/\fq) \in (x(R/\fq) : K(R/\fq))$ as a reduction for $\fm/\fq$. Then by Lemma \ref{LemSup}, there exists a canonical ideal $I$ of $R$ such that $x_1 \in I$. As $x_1\bar{R} = x_1(R/\fq)$ is a reduction for $\fm/\fq$ and $\fq$ generates by a superficial sequence for $\fm$, $x_1, x_2, \cdots , x_d$ is a superficial sequence for $\fm$ of length $d$, and so $x_1R + \fq$ is a minimal reduction of $\fm$ and $I$ is a canonical reduction of $R$.      
				
				(ii).  Choose an $S$-regular sequence $\fy = y_1, \cdots , y_t$ such that (1) $\fy$ is a superficial sequence for $\fn$ and (2) $\fy(S/\fm S)$ is a regular system of parameter for $\fn/\fm S$. Then $\e_{\fn}^0(S) = \e_{\fn/\fy S}^0(S/\fy S)$ and $\ell_S(S/(\fm S+\fy S)) = \e_{\fn/\fm S}(S/\fm S)$. Therefore, as the composition map $R \overset{\varphi}{\lo} S \lo S/\fy S$ is flat by \cite[Lemma 1.23]{HK} and $\dim(R) = \dim(S/\fy S)$, the result follows by (i) and (ii). 
						\end{proof}
						
					Next is an immediate corollary of Proposition \ref{canredflat}. 
					
			\begin{cor} Assume that $(R,\fm)$ is a $d$-dimensional Cohen-Macaulay local ring, admitting canonical ideal, with infinite residue field. Then
				\begin{itemize}
					\item [(a)] Let $n \in \mathbb{N}$ and $X_1, X_2, \cdots, X_n$ be indeterminate. Then,  $R$ has a canonical reduction if and only if $R[[X_1, X_2, \cdots, X_n]]$ has a canonical reduction.
					\item [(b)] Let $\hat{R}$ denotes the $\fm$-adic completion of $R$. Then, $R$ has a canonical reduction if and only if $\hat{R}$ has a canonical reduction.
				\end{itemize}
			\end{cor}

			{\bf Acknowledgment.}  I would like to thank Professor Shiro Goto for presenting some lectures about `` almost Gorenstein rings " in Iran (2014). The author thanks the Graduate School of Advanced Mathematical Sciences of Meiji University hosting him as a visitor and a member of Professor Goto research group. The author also thanks Raheleh Jafari for  discussions on numerical semigroup rings.

\end{document}